\documentclass[a4paper, 11pt, parskip=half]{amsart}
\usepackage{custom-template, shortcuts}

\begin{document}
	\title{Distinguishing Siegel modular forms}
	\date{}
	\author{Arvind Kumar and Ariel Weiss}
	\address{Arvind Kumar, Department of Mathematics, Indian Institute of Technology Jammu, Jagti, PO Nagrota, NH-44 Jammu 181221, J \& K, India\vspace*{-3pt}}
	\email{arvind.kumar@iitjammu.ac.in\vspace*{-6pt}}
	\address{Ariel Weiss, Department of Mathematics, The Ohio State University, Columbus, OH, USA.\vspace*{-14pt}}
\address{Department of Mathematics, Trinity College, Hartford, CT, USA}
\email{ariel.weiss@trincoll.edu\vspace*{-3pt}}
	\begin{abstract}

Let $f$ and $f'$ be genus $2$ cuspidal Siegel paramodular newforms. We prove that if their Hecke eigenvalues $a_p$ and $a_p'$ satisfy a non-trivial polynomial relation $P(a_p, a_p') = 0$ for a set of primes $p$ of positive density, then $f$ is a scalar multiple of a quadratic twist of $f'$. This result extends the strong multiplicity one theorem, which handles the case $P(x,y) = x - y$, to arbitrary polynomial relations.

Our proof analyses the image of the product Galois representation attached to the pair $(f, f')$: we show that this image is as large as possible, unless $f$ is a twist of $f'$.

Our results also apply to elliptic modular forms. They therefore provide a unified method for distinguishing both elliptic and Siegel modular forms based on their Hecke data, including their Hecke eigenvalues, Satake parameters, Sato--Tate angles, and the coefficients of their $L$-functions. We apply our methods to recover and generalise a range of existing results and to prove new ones in both the elliptic and Siegel settings.
	\end{abstract}
	
	\maketitle
	\section{Introduction}

Let $f$ and $f'$ be genus $2$ cuspidal Siegel paramodular newforms. Suppose that for almost all primes $p$, their Hecke eigenvalues $a_p$ and $a_p'$ satisfy a fixed polynomial relation $P(a_p, a_p') = 0$.
What is the relationship between $f$ and $f'$?

If $P(x, y) = x - y$, then this question is answered by the \emph{strong multiplicity one} theorem \cite{schmidt,schmidt_acta,wang,farmer}, which states that $a_p = a_p'$ for almost all primes $p$ if and only if $f$ is a scalar multiple of $f'$. But for other polynomials, this conclusion fails. For example, a form and its quadratic twists have Hecke eigenvalues satisfying $a_p^2 = {a_p'}^2$ for all but finitely many $p$, but are not scalar multiples.

In this paper, we show that twists are the only possible exceptions: if $a_p$ and $a_p'$ satisfy a non-trivial polynomial relation for a positive proportion of primes, then $f$ and $f'$ must be quadratic twists of each other by a Dirichlet character. This result can be viewed as a refinement of the strong multiplicity one theorem: even partial algebraic information about a form's Hecke eigenvalues is sufficient to determine it up to twist.

We now state our result more precisely and in slightly more generality. Let $\pi$ and $\pi'$ be cuspidal automorphic representations of $\Gf(\AQ)$ whose archimedean components are holomorphic discrete series or limit of discrete series representations with Blattner parameters $(k_1, k_2)$ and $(k_1', k_2')$. Let $S, S'$ be the sets of primes at which $\pi,\pi'$ are ramified. Denote their central characters by $\epsilon$ and $\epsilon'$, and their $T_p$-Hecke eigenvalues by $a_p$ and $a_p'$. Let $\Pi,\Pi'$ be the transfers of these representations to $\GL_4(\AQ)$. We assume that these transfers exist, are cuspidal, and are neither automorphic inductions nor symmetric cube lifts. These automorphic representations correspond to Siegel modular forms of general type (type \textbf{(G)} in the notation of \cite{schmidt}) with weights $(k_1, k_2)$ and $(k_1', k_2')$---i.e.\ weights $\Sym^{k_1 - k_2}\det^{k_2}$ and $\Sym^{k_1' - k_2'}\det^{k_2'}$---with $k_1\ge k_2\ge 2$ and $k_1'\ge k_2'\ge 2$. In particular, our results apply to non-lift, non CM/RM, type \textbf{(G)} paramodular newforms $f$, $f'$ of levels $K(N)$ and $K(N')$, where $K(N)$ denotes the paramodular group of level $N$.

\begin{theorem}\label{cor:trace}
Let $P(a,a')\in \Qb[a, a']$ be any non-zero polynomial. Suppose that for a set of primes $p$ of positive upper density
\[P( a_p, a_p')=0.\] 
Then $(k_1, k_2) = (k_1', k_2')$ and there is a Dirichlet character $\chi$ such that $\Pi\simeq\Pi'\otimes\chi$. Hence, if $f,f'$ are paramodular newforms, then $f$ is a scalar multiple of a quadratic twist of $f'$.
\end{theorem}

\begin{remark}\label{rem:twisting}
What we actually prove is that there exists a Dirichlet character $\chi$ such that $a_p = \chi(p)a_p'$ for all but finitely many primes $p$. By the strong multiplicity one theorem for $\GL_4$, it follows that $\Pi \simeq \Pi' \otimes \chi$. If $\pi$ and $\pi'$ have trivial central character, which is always the case for paramodular newforms, then $\chi$ must be quadratic (see \Cref{rem:quadratic}). 

For general automorphic representations $\pi$ and $\pi'$, the conclusion $\Pi \simeq \Pi' \otimes \chi$ is the strongest possible, since the strong multiplicity one theorem is false for automorphic representations of $\Gf(\AQ)$. However, if $\pi$ and $\pi'$ arise from paramodular newforms $f$ and $f'$, then \cite{wang}*{Lemma~4.1} gives $\pi \simeq \pi' \otimes (\chi\circ\simil)$. Hence, by \cite{schmidt}*{Thm.~2.6(ii)}, $f$ is a scalar multiple of $f'\otimes\chi$, where $f'\otimes\chi$ denotes the paramodular newform associated to $\pi'\otimes(\chi\circ\simil)$.
\end{remark}

More generally, we show that an algebraic relation among \emph{any} of the Hecke eigenvalues of a Siegel paramodular newform $f$ is sufficient to determine it up to twist. The Hecke algebra of $\pi$ is generated by $T(p)$,  $T(p^2)$, and the diamond operator $\Delta(p)$. Its $T(p)$-Hecke eigenvalues $a_p$ are the coefficients of its \emph{spinor $L$-function}
\begin{align}\label{eq:l-function}
   L_{\mathrm{spin}}(\pi, s) & =\prod_p L_{p,\mathrm{spin}}(\pi, s) = L(\epsilon, 2s-u+1)\sum_{n\ge 1}a_nn^{-s},
\end{align}
where $u = k_1 + k_2 -3$, $\epsilon$ is the central character of $\pi$, and
\[L_{p,\mathrm{spin}}(\pi, s) =(1 - a_pp^{-s} + (a_p^2 -a_{p^2} -\epsilon(p)p^{u-1})p^{-2s} - a_p\epsilon(p)p^{u}p^{-3s} + \epsilon(p^{2})p^{2u}p^{-4s})\ii\]
for an unramified prime $p$ (see \cite{andrianov}*{(5.64)} and \cite{brownkeaton}*{p.~1499}).

Let
\begin{equation}\label{b_p}
b_p = \frac{a_p^2 -a_{p^2}}{\epsilon(p)p^{u}}-\frac1p-1.
\end{equation}
Then $b_p$ satisfies the equation
\[(b_p + 1)\epsilon(p)p^u = a_p^2 -a_{p^2} -\epsilon(p)p^{u-1}\]
and is the coefficient of the \emph{standard} $L$-function of $\pi$,
\[L_{\std}(\pi, s) = \sum_{n\ge 1}b_nn^{-s}.\]

We prove the following more general result, from which \Cref{cor:trace} follows immediately:
\begin{theorem}\label{thm:b_p}
				Let
		$P(s, s', a, b, a', b')\in  \Qb[s, \frac1s, s',\frac1{s'},a, b,a',b']$ be a non-zero polynomial. Assume that $P$ is coprime to $s^\kappa - {s'}^{\kappa'}$, where $\kappa,\kappa'$ are integers defined in \eqref{def:kappa}.
        Suppose that for a set of primes $p$ of positive upper density
\[P(\epsilon(p)p^{k_1 + k_2 - 3},\epsilon'(p)p^{k_1' + k_2' - 3},a_p, b_p,  a_p', b_p')=0.\] 
Then $(k_1, k_2) = (k_1', k_2')$ and there is a Dirichlet character $\chi$ such that $\Pi\simeq\Pi'\otimes\chi$. Hence, if $f,f'$ are paramodular newforms, then $f$ is a scalar multiple of a quadratic twist of $f'$.
\end{theorem}

We note that the coprimality condition on $P$ is necessary to rule out trivial polynomial relations that can arise from the weights and characters of $\pi$ and $\pi'$, in which case the conclusion is false.

This result strengthens the observation in \cite{pitale-book,farmer} on the strong multiplicity one theorem, where the authors note that it is ``remarkable'' that the eigenvalues $a_p$ alone suffice to determine a paramodular newform $f$, despite the fact that the Hecke algebra is not generated by $T(p)$. We show that an algebraic relation involving \emph{any} of the Hecke data for just a positive proportion of primes suffices to determine the form up to twist. 

\begin{remark}
Throughout this paper, we assume the automorphic representations $\pi$ and $\pi'$ have cuspidal functorial lifts to $\GL_4(\AQ)$. This assumption holds if one assumes Arthur’s endoscopic classification \cite{Arthur2013}. However, this classification depends on a number of unpublished results, of which a portion have appeared \cite{geetaibi,arthur-plug}.  When $k_2 >2$, these functorial lifts are known to exist unconditionally thanks to work of Weissauer \cite{Weissauersymplectic} and Asgari--Shahidi \cite{asgari-shahidi}. However, in the case $k_2 = 2$, the existence of such a lift depends on Arthur’s results. 
For further discussion, see \cite{calegari-blog}.
\end{remark}

In \Cref{sec:applications}, we demonstrate how \Cref{cor:trace,thm:b_p} can be used to prove a broad range of results that distinguish Siegel modular forms by their Hecke data. In addition to proving new results, we strengthen the results of \cites{kms, kumar-meher-shankhadar, wang} and extend the scope of previous results from the elliptic modular setting to the setting of Siegel modular forms \cites{raj, Ramakrishnan-appendix, rajan-forms-squares,kulkarni,murty-pujahari, para, gun-murty-paul, para2,wong}. We emphasise that both theorems and their applications easily generalise to the case of elliptic modular forms, with analogous proofs (see \Cref{sec:modular}). In the next two sections, we outline the proof of \Cref{thm:b_p}.

\subsection{Galois representations attached to Siegel modular forms}

To prove \Cref{thm:b_p}, we reinterpret algebraic relations between Hecke eigenvalues as relations occurring in the images of Galois representations attached to Siegel modular forms.

For each prime $\l$, there exists a semisimple Galois representation
\[\rho_\l\:\Ga\Q\to\Gf(\Qlb)\]
associated to $\pi$. Let $S$ be the set of primes where $\pi$ is ramified. Then $\rho_\l$ is unramified outside $\l$ and $S$, and if $p\notin \{\l\}\cup S$, then the characteristic polynomial of $\rho_\l(\Frob_p)$ agrees with the Euler factor of $L_{\mathrm{spin}}(\pi, s)$ at $p$. In particular, we have
\[a_p = \Tr\rho_\l(\Frob_p),\]
\[\epsilon(p)p^{k_1 + k_2 -3}=\simil\rho_\l(\Frob_p) ,\]
and
\[b_p = \Tr\std\rho_\l(\Frob_p) = \frac{\tr(\wedge^2\rho_\l(\Frob_p))}{\simil\rho_\l(\Frob_p)}-1 ,\]
where $\simil\:\Gf\to\GG_m$ is the similitude character, and $\std\:\Gf\to \PGSp_4\cong\SO_5\hookrightarrow\GL_5$ is the standard representation (via the exceptional isomorphism $\PGSp_4\cong\SO_5$).
We similarly define
\[\rho_\l'\:\Ga\Q\to\Gf(\Qlb)\]
attached to $\pi'$, and consider the product representation
\[R_\l = \rho_\l\times\rho_\l'\:\Ga\Q\to\Gf(\Qlb)\times\Gf(\Qlb).\]
The vanishing of a polynomial $P$ as in \Cref{thm:b_p} for a set of primes of positive density implies the vanishing of the function
\[(\gamma, \gamma')\mapsto P(\simil(\gamma), \simil(\gamma'), \tr(\gamma),  \tr\std(\gamma),\tr(\gamma'), \tr\std(\gamma'))\]
on a large part of the image of $R_\l$. We prove \Cref{thm:b_p} by studying the image of $R_\l$.

\subsection{Images of Galois representations}

Let $\chi_{\mathrm{cyc},\l}$ denote the $\l$-adic cyclotomic character. The fact that
\[\simil\rho_\l\simeq\epsilon\chi_{\mathrm{cyc},\l}^{k_1+k_2-3}\qquad\text{and}\qquad\simil\rho_\l'\simeq\epsilon'\chi_{\mathrm{cyc},\l}^{k_1'+k_2'-3}\]
ensures that the image of $R_\l$ lies inside the subgroup
\[\Gkk(\Qlb) = \set{(\gamma, \gamma')\in \Gf(\Qlb)\times\Gf(\Qlb) : \simil(\gamma)^{\kappa} = \simil(\gamma')^{\kappa'}},\]
where 
$\kappa$ and $\kappa'$ are the smallest positive integers such that 
\[(\epsilon(p)p^{k_1 + k_2 -3})^{\kappa} = (\epsilon'(p)p^{k_1' + k_2' -3})^{\kappa'}\]
for all primes $p\notin S\cup S'$. Explicitly, we have
\begin{equation}\label{def:kappa}
(\kappa, \kappa') = \ord\br{\frac{\epsilon^{(k_1' + k_2'-3)/n}}{(\epsilon')^{(k_1 + k_2-3)/n}}}\cdot\br{\frac{k_1' + k_2'-3}{n},\frac{k_1 + k_2-3}{n}},
\end{equation}
where 
\[n=\gcd(k_1 + k_2 - 3, k_1' + k_2' -3).\]
We view $\Gkk$ as an affine algebraic group over $\Spec\Qb$: it is the fibre product of $\Gf$ with itself, fibred over the similitude maps $\simil^{\kappa}$ and $\simil^{\kappa'}$. In particular, we have
\[\Gkk = \Spec\O(\Gkk),\]
where $\O(\Gkk)$ is the ring of global functions of $\Gkk$. In \Cref{sec:gkk}, we give an explicit description of the $\Qb$-algebra $\O(\Gkk)$ and of $\Gkk$'s connected components.

Recall that $\pi$ and $\pi'$ are of general type, and that neither has CM, RM, or is a symmetric cube lift. By work of Dieulefait \cite{Dieulefait2002maximalimages}, Dieulefait--Zenteno \cite{DZ}, and the second author \cite{weissthesis,weiss2018image}, the images of $\rho_\l$ and $\rho_\l'$ are generically large. In particular, there is a set of primes $\LL$ of Dirichlet density $1$ such that if $\l\in\LL$, then the Zariski closures of the images of $\rho_\l$ and $\rho_\l'$ in ${\Gf}_{/\Qlb}$ are $\Gf(\Qlb)$ (\Cref{cor:zariski-closed}). 

Let $\Gamma_\l$ denote the Zariski closure of the image of $R_\l$ in ${\Gkk}_{/\Qlb}$. Using the fact that both projections of $\Gamma_\l$ onto $\Gf(\Qlb)$ are surjective, we show that $\Gamma_\l$ is strictly contained in $\Gkk(\Qlb)$ if and only if $\pi$ and $\pi'$ are related. More precisely, we prove:

\begin{theorem}\label{thm:joint-large-image}
    Exactly one of the following two conditions holds:
    \begin{enumerate}
        \item For all primes $\l\in\LL$, we have $\Gamma_\l = \Gkk(\Qlb)$.
        \item There exists a Dirichlet character $\chi$ such that $\Pi \simeq \Pi'\tensor\chi$. 
    \end{enumerate}
\end{theorem}

To deduce \Cref{cor:trace,thm:b_p}, it remains to relate the existence of a polynomial $P$ satisfying the vanishing hypothesis to the strictness of the inclusion $\Gamma_\l\sub\Gkk(\Qlb)$.

Let $\O(\Gkk)^{\Gkk}$ denote the subring of $\O(\Gkk)$ consisting of functions that are invariant under the action of $\Gkk$ on itself by conjugation. Note that if $P$ is a polynomial as in \Cref{thm:b_p}, then the function
\begin{equation}\label{eq:phi}
\phi\colon(\gamma, \gamma')\mapsto P(\simil(\gamma), \simil(\gamma'), \tr(\gamma), \tr\std(\gamma),\tr(\gamma'),  \tr\std(\gamma'))
\end{equation}
is an element of $\O(\Gkk)^{\Gkk}$.

By combining \Cref{thm:joint-large-image} with Rajan's algebraic version of the Chebotarev density theorem \cite{raj}, we prove the following theorem:

\begin{theorem}\label{thm:galois}
    Let $\phi$ be an element of $\O(\Gkk)^{\Gkk}$ that does not vanish on any connected component of $\Gkk$.
    Suppose that for a set of primes $p$ of positive upper density, we have
    \[\phi(R_\l(\Frob_p)) = 0.\]
Then there exists a Dirichlet character $\chi$ such that $\Pi\simeq\Pi'\otimes\chi$. Hence, if $f,f'$ are paramodular newforms, then $f$ is a scalar multiple of a quadratic twist of $f'$.
\end{theorem}

\Cref{cor:trace,thm:b_p} follow from \Cref{thm:galois}. Indeed, we show that the hypothesis in \Cref{thm:b_p} that $P$ is coprime to $s^\kappa - {s'}^{\kappa'}$ exactly ensures that the corresponding function $\phi$ defined in \eqref{eq:phi} does not vanish on any connected component of $\Gkk$.

\subsection{Results for a single Siegel modular form}

While our results concern relations between pairs of Siegel modular forms, they can also be used to prove results about the Hecke eigenvalues of a single Siegel modular form.

Suppose that $\pi$ is a cuspidal automorphic representation of $\Gf(\AQ)$ satisfying the hypotheses of \Cref{cor:trace}.

\begin{theorem}\label{thm:one-form}
        Let $P(s,a,b)\in\Qb[s,\frac1s,a,b]$ be a non-zero polynomial. Then the set of primes
    \[\set{p : P(\epsilon(p)p^{k_1 + k_2 - 3},a_p,b_p) = 0}\]
     has Dirichlet density $0$. 
\end{theorem}

Note that \Cref{thm:one-form} follows immediately from \Cref{thm:b_p} by viewing $P$ as an element of $\Qb[s,\frac1s,s',\frac1{s'},a,b,a',b']$ and choosing $\pi'$ such that $\Pi$ is not a twist of $\Pi'$.

It is also possible to prove \Cref{thm:one-form} directly, without appealing to \Cref{thm:b_p}, by applying our techniques to a single Siegel modular form. We include \Cref{thm:one-form} here to emphasise that \Cref{thm:b_p} captures such statements uniformly.

\subsection{Results for elliptic modular forms}\label{sec:modular}

Our methods equally apply to elliptic modular forms. In this setting, we replace \Cref{thm:joint-large-image} with the classical results of Ribet and Momose on the images of Galois representations attached to modular forms \cite{Ribet77, momose, Ribet85,Loeffler-adelic}. We state the corresponding analogue of \Cref{thm:b_p} in this context and note that our applications in \Cref{sec:applications} all have analogues in this setting.

Let $f$ and $f'$ be cuspidal newforms of weights $k, k' \ge 2$, levels $N, N'$, and nebentypus characters $\epsilon, \epsilon'$. For each prime $p$, let $a_p$ and $a_p'$ denote their Hecke eigenvalues. Assume that $f$ and $f'$ do not have complex multiplication. Similarly to \eqref{def:kappa}, define $\kappa$ and $\kappa'$ to be the smallest positive integers such that
\[(\epsilon(p)p^{k-1})^\kappa = (\epsilon'(p)p^{k'-1})^{\kappa'}\]
for all primes $p\nmid NN'$. Then we obtain the following result:

\begin{theorem}\label{thm:modular}
    Let $P(s, s', a, a')\in\Qb[s,\frac1s,s',\frac1{s'},a,a']$. Assume that $P$ is non-zero and coprime to $s^\kappa - {s'}^{\kappa'}$. 
    Suppose that for a set of primes $p$ of positive upper density
    \[P(\epsilon(p)p^{k-1}, \epsilon'(p)p^{k'-1},a_p,a_p')=0.\]
    Then $k = k'$ and there is a Dirichlet character $\chi$ such that $f = f'\tensor\chi$.
\end{theorem}

\begin{corollary}
    Let $P(s,a)\in\Qb[s,\frac1s,a]$ be a non-zero polynomial. Then the set of primes
     \[\set{p : P(\epsilon(p)p^{k-1},a_p) = 0}\]
    has density $0$.
\end{corollary}

For any integer $n\ge 1$, the Hecke eigenvalue $a_{p^n}$ can be expressed as a polynomial in $a_p$ and $\epsilon(p)p^{k-1}$. Hence, as a corollary to \Cref{thm:modular}, we immediately deduce the following result:

\begin{corollary}\label{cor:apn}
    Fix $n\iN$ and suppose that $a_{p^n} = a'_{p^n}$ for a set of primes $p$ of positive upper density. Then $k = k'$ and there is a Dirichlet character $\chi$ such that $f = f'\tensor\chi$.
\end{corollary}

In fact, we can show that if $P\in\Qb[a,a']$ is any non-zero polynomial such that $P(a_{p^n}, a'_{p^n})= 0$ for a positive density of $p$, then $f$ is a twist of $f'$ (see \Cref{rmk:symmetric}).

\subsection{Convention and Notation}
Throughout this paper, we adopt the following conventions and notational framework:
\begin{itemize}
    \item $\pi$ and $\pi'$ denote cuspidal automorphic representations of $\Gf(\AQ)$, whose archimedean components are holomorphic discrete series or limit of discrete series representations. We refer to \cite{weiss2018image}*{Sec.~2} for more precise definitions, including that of the Blattner parameter $(k_1, k_2)$ of a (limit of) discrete series representation. If $\pi$ corresponds to a Siegel modular form $f$, then this parameter matches the weight of $f$. We let $\epsilon$ and $\epsilon'$ denote the central characters of $\pi$ and $\pi'$. We let $S$ and $S'$ denote the sets of primes at which $\pi$ and $\pi'$ are ramified.
    \item Let $f $ and $f'$ denote cuspidal \emph{Siegel paramodular newforms} of genus 2 with weights $(k_1, k_2)$ and  $(k_1', k_2')$---i.e.\ weights $\Sym^{k_1 - k_2}\det^{k_2}$ and $\Sym^{k_1' - k_2'}\det^{k_2'}$---with $k_1\ge k_2\ge 2$ and $k_1'\ge k_2'\ge 2$, and levels $K(N)$ and $K(N')$, where $K(N)$ is the paramodular group of level $N$. In this case, we follow the conventions of \cite{schmidt}.

    \item We assume throughout that $\pi$ and $\pi'$ are of general type (type \textbf{(G)} in the notation of \cite{schmidt}), and that neither has CM, RM, nor arises as a symmetric cube lift.
    \item For each prime $p $, let $a_p $ and $a_p' $ denote the \emph{Hecke eigenvalues} of $\pi $, $\pi' $ with respect to the standard Hecke operator $T_p $. Similarly, let $b_p $ and $b_p' $ denote the coefficients of the \emph{standard $L $-function}.

    \item Let $\kappa $, $\kappa' $ be the smallest positive integers such that
    \[
    (\epsilon(p)p^{k_1 + k_2 - 3})^\kappa = (\epsilon'(p)p^{k_1' + k_2' - 3})^{\kappa'}
    \]
    for all primes $p\notin S\cup S'$. We use $\kappa$ and $\kappa'$ to relate the similitude characters of the Galois representations attached to $\pi$ and $\pi'$.

    \item If $f$ is a paramodular newform and $\chi$ is a quadratic character, we denote by $f \otimes \chi $ the \emph{quadratic twist} of $f$: the paramodular newform,
    unique up to scalar, associated to $\pi\otimes(\chi\circ\simil)$. Its existence and uniqueness follow from \cite{schmidt}*{Thms.~1.1(ii), 2.6} and
    \cite{roberts-schmidt}*{Thm.~7.5.4}. See \cite{jlr} for explicit twisting operators. We have
    \[
    a_p(f\otimes\chi) = \chi(p)a_p(f)
    \]
    for all but finitely many primes $p$.
\end{itemize}

\section{Applications: New Results and Extensions}\label{sec:applications}

In this section, we present numerous applications of \Cref{cor:trace,thm:b_p}.

We note that, by \Cref{rem:twisting}, if $f$ and $f'$ are paramodular newforms, then in all the results of this section, the conclusion $\Pi\simeq\Pi'\tensor\chi$ can be replaced with the conclusion that $f$ is a scalar multiple of a quadratic twist of $f'$. We also note that all the results stated in this section have analogues for elliptic modular forms (cf.~\Cref{sec:modular}). 

\subsection{Distinguishing eigenforms by their Hecke eigenvalues}

Let $n$ and $m$ be positive integers. Taking $P(s,s',a,b,a',b') = (a^n - {a'}^n)(b^m-{b'}^m)$ in \Cref{thm:b_p}, we immediately deduce the following corollary:

\begin{corollary}\label{cor:apbp}
    Suppose that for a set of primes $p$ of positive upper density, either $a_p^n = {a_p'}^n$ or $b_p ^m= {b_p'}^m$. Then $(k_1, k_2) = (k_1', k_2')$ and there is a Dirichlet character $\chi$ such that $\Pi \simeq\Pi'\tensor\chi$.
\end{corollary}

\begin{remarks}
    \begin{enumerate}
        \item If $f$ and $f'$ are paramodular newforms of level $1$, then $\chi$ must be trivial and we conclude that $f$ is a scalar multiple of $f'$.  Thus, \Cref{cor:apbp} generalises \cite{kms}*{Thm.~1.5} to Siegel modular forms of arbitrary level.
        \item If $f$ and $f'$ are paramodular newforms, then the character $\chi$ must be quadratic, and we deduce from \Cref{rem:twisting} that $f$ is a scalar multiple of $f'\tensor\chi$. Thus, \Cref{cor:apbp} generalises \cite{wang}*{Thm.~1.2} to include the case $k_2 = 2$ (cf.\ \cite{wang}*{Rem.~2}) and weakens the hypothesis by requiring only that $b_p = b_p'$ for a positive density of primes, rather than for almost all primes.
        \item The analogue of \Cref{cor:apbp} for elliptic modular forms recovers \cite{raj}*{Cor.~1} when $n = 1$, the main result of  \cite{Ramakrishnan-appendix} when $n=2$, and the main result of \cite{rajan-forms-squares} in general.
    \end{enumerate}
\end{remarks}

We also deduce the following result, which is the analogue of \cite{para2}*{Cor.~1.1} for Siegel modular forms (see also \cite{kulkarni}*{Thm.~2}):

\begin{corollary}\label{cor:arbitrary power}
    Suppose that for a positive density of primes $p$, there exists a positive integer $n_p$ such that $a_p^{n_p} = {a_p'}^{n_p}$. Then $(k_1, k_2) = (k_1', k_2')$ and there is a Dirichlet character $\chi$ such that $\Pi\simeq\Pi'\tensor\chi $.
\end{corollary}

\begin{proof}
    If $a_p =a_p' = 0$ for a positive density of primes, then the result follows from \Cref{cor:apbp}. Thus we may assume that $a_p^{n_p} = {a_p'}^{n_p}\ne 0$ for a positive density of primes $p$. For such $p$, if $a_p^{n_p} = {a_p'}^{n_p}$, then $a_p/a_p'$ is a root of unity. Let $E$ be the number field containing all the $a_p$'s and $a_p'$'s. Since $E$ contains only finitely many roots of unity, it follows that there is an integer $n$ such that $a_p^n = {a_p'}^n$ for a positive density of primes $p$. The result follows from \Cref{cor:apbp}.
\end{proof}

\subsection{Distinguishing eigenforms by their normalised eigenvalues}\label{sec:eigenvalues}

For each prime $p$ and each positive integer $n$, let 
\[\lambda_{p^n} = \frac{a_{p^n}}{p^{n(k_1 + k_2 - 3)/2}}\]
and
\[\lambda_{p^n}' = \frac{a_{p^n}'}{p^{n(k'_1 + k'_2 - 3)/2}}\]
be the normalised $T_{p^n}$-Hecke eigenvalues of $\pi$ and $\pi'$.

\begin{theorem}\label{cor:normalised-trace}
				  Let
	$P(x,y)\in \overline \Q[x, y]$ be any  non-zero polynomial, and suppose that for a set of primes $p$ of positive upper density
\[P( \lambda_p, \lambda_p')=0.\] 
Then $(k_1, k_2) = (k_1', k_2')$ and there is a Dirichlet character $\chi$ such that $\Pi\simeq\Pi'\otimes\chi$.
\end{theorem}

We will make frequent use of the following lemma:

\begin{lemma}\label{lem:divis}
    Let $R$ be a $\Qb$-algebra, let $P(x)\in R[x]$ be a non-zero polynomial, and let $d$ be a positive integer. Assume that $R$ is an integral domain. Let 
    \[F(x) = \prod_{\zeta\in\mu_d}P(\zeta x).\]
    Then there exists a non-zero polynomial $Q$ such that $F(x) = Q(x^d)$.
    In particular, there exists a non-zero polynomial $Q\in R[x]$ such that $P(x)\mid Q(x^d)$.
\end{lemma}
\begin{proof}

    Clearly $P(x)\mid F(x)$. Since $F(\zeta x) = F(x)$ for all $\zeta\in\mu_d$, it follows that $F(x) = Q(x^d)$ for some polynomial $Q(x)\in R[x]$. Since $P(\zeta x)$ is non-zero for all $\zeta$, $F(x)$ is non-zero, and hence $Q$ is non-zero.
\end{proof}

\begin{proof}[Proof of \Cref{cor:normalised-trace}]
Let $d$ be an integer such that $\epsilon^{d} = {\epsilon'}^{d} = 1$. By the construction of \Cref{lem:divis}, there exists a polynomial $Q\in \Qb[x,y]$ such that $P(x,y)\mid Q(x^{2d}, y^{2d})$.

Let $\mathcal P(s, s', a, b, a', b') = Q(\frac{a^{2d}}{s^d}, \frac{{a'}^{2d}}{{s'}^d})\in\Qb[s,\frac1s,s',\frac1{s'},a,b,a',b']$. 
Since 
\[\lambda_p^{2d} = \frac{a_p^{2d}}{p^{d(k_1 + k_2-3)}} = \frac{a_p^{2d}}{(\epsilon(p)p^{k_1 + k_2-3})^d}\] 
and similarly for ${\lambda_p'}^{2d}$, we see that 
\[\mathcal P(\epsilon(p)p^{k_1 + k_2 - 3}, \epsilon'(p)p^{k_1'+k_2' - 3}, a_p, b_p, a_p', b_p') = Q(\lambda_p^{2d}, {\lambda'}_p^{2d}) = 0\]
for a positive proportion of $p$. 

Now, $\mathcal{P}$ is coprime to $s^\kappa - {s'}^{\kappa'}$. Indeed, if $\mathcal{P}$ had a common factor with $s^\kappa - {s'}^{\kappa'}$, then $\mathcal{P}(s^2, {s'}^2, a,b,a',b')$ would have a common factor with $s^{2\kappa} - {s'}^{2\kappa'}$. But 
\[\mathcal{P}(s^2, {s'}^2, a,b,a',b') = Q((a/s)^{2d}, (a'/s')^{2d}),\]
and by the construction in \Cref{lem:divis}, the right-hand side is a product of non-zero polynomials of the form $P\left( \zeta (a/s), \zeta'(a'/s')\right)$ with $\zeta,\zeta'\in \mu_{2d}$. After the change of variables $a \mapsto s a$, $a' \mapsto s' a'$, each of these polynomials becomes independent of $s$ and $s'$, and thus coprime to $s^{2\kappa} - {s'}^{2\kappa'}$. Thus, $\mathcal{P}$ is coprime to $s^\kappa - {s'}^{\kappa'}$, and the result follows from \Cref{thm:b_p}.
\end{proof}

Taking $P(x, y) = x^n - y^n$ in \Cref{cor:normalised-trace}, we immediately deduce the following result:

\begin{corollary}\label{cor:power}
    Let $n$ be a positive integer, and suppose that for a set of primes $p$ of positive upper density, $\lambda_p^n = {\lambda_p'}^n$. Then there is a Dirichlet character $\chi$ such that $\Pi\simeq\Pi'\tensor\chi.$
\end{corollary}

The analogue of \Cref{cor:power} for elliptic modular forms recovers \cite{murty-pujahari}*{Thm.~1} and \cite{para}*{Thm.~1} when $n 
=1$.

\subsection{Distinguishing eigenforms by the absolute values of their Hecke eigenvalues}
\begin{theorem}\label{cor:absolute}
				 Let
		$P(s,s',a, b, a', b')\in  \Qb[s,\frac 1s,s',\frac{1}{s'},a, b,a',b']$ be a non-zero polynomial. 
        Assume that $P$ is coprime to $s^\kappa - {s'}^{\kappa'}$.
        Suppose that for a set of primes $p$ of positive upper density
\[P(p^{k_1 + k_2 - 3}, p^{k_1' + k_2'-3},|a_p|, |b_p|,  |a_p'|, |b_p'|)=0.\] 
Then $(k_1, k_2) = (k_1', k_2')$ and there is a Dirichlet character $\chi$ such that $\Pi\simeq\Pi'\otimes\chi$. 
\end{theorem}

\begin{proof}
Applying the properties of the Petersson inner product (\cite{anzh}*{Prop.~1.8}), we have
 \[
 a_p=\epsilon(p)\overline{a_p} \quad   {\rm and} \quad a_{p^2}=\epsilon(p^2)\overline{a_{p^2}},
 \]
where $\overline z$ denotes the complex conjugate of the complex number $z$. Thus, by \eqref{b_p}, we have
\[\overline{b_p} = \frac{\overline{a_p}^2 - \overline{a_{p^2}}}{\overline{\epsilon(p)}p^u}-\frac1p - 1 = b_p,\]
so $b_p\iR$.
Set
\[\zeta_p = \begin{cases}
    \frac{a_p}{|a_p|}&a_p\ne 0\\
    1&a_p = 0,
\end{cases}\qquad \delta_p = \begin{cases}
    \frac{b_p}{|b_p|}&b_p\ne 0\\
    1&b_p = 0.
\end{cases}\]
Then $\delta_p\in\{\pm1\}$, and since $a_p = \epsilon(p)\overline{a_p}$, we have $\zeta_p^2= \epsilon(p)$ when $a_p\ne 0$. Thus $\zeta_p$ can take only finitely many values. Define $\zeta_p',\delta_p'$ analogously. By the pigeonhole principle, there exists a fixed tuple $(e, e',\zeta,\delta,\zeta',\delta')$ and a set of primes $p$ of positive upper density for which both $P(p^{k_1 + k_2 - 3}, p^{k_1' + k_2'-3},|a_p|, |b_p|,  |a_p'|, |b_p'|)=0$ and
\[(\epsilon(p), \epsilon'(p), \zeta_p, \delta_p, \zeta_p', \delta_p') = (e,e',\zeta,\delta,\zeta',\delta').\]

Set 
\[P_0(s, s', a,b,a',b') = P(e\ii s, {e'}\ii s', \zeta\ii a, \delta b, {\zeta'}\ii a', \delta'b').\]
Then for primes $p$ as above, 
\begin{align*}
    P_0(\epsilon(p)p^{k_1+k_2-3}, \epsilon'(p)p^{k_1'+k_2'-3}, a_p, b_p, a_p', b_p')&=P(p^{k_1 + k_2 - 3}, p^{k_1' + k_2'-3},|a_p|, |b_p|,  |a_p'|, |b_p'|) = 0.
\end{align*}
Moreover, since $e^{\kappa} = {e'}^{\kappa'}$, the change of variables defining $P_0$ sends $s^\kappa - {s'}^{\kappa'}$ to $e^{-\kappa}(s^\kappa - {s'}^{\kappa'})$. Thus $P_0$ remains non-zero and coprime to $s^\kappa - {s'}^{\kappa'}$. The result follows from \Cref{thm:b_p}.
\end{proof}

\begin{corollary}\label{cor:absolute_normalized}
				 Let
		$P(x, y)\in  \Qb[x,y]$ be a non-zero polynomial.  Suppose that for a set of primes $p$ of positive upper density
\[P(|\lambda_p|, |\lambda_p'|)=0.\] 
Then $(k_1, k_2) = (k_1', k_2')$ and there is a Dirichlet character $\chi$ such that $\Pi\simeq\Pi'\otimes\chi$. 
\end{corollary}
\begin{proof}
 By \Cref{lem:divis}, there exists a polynomial $Q\in \Qb[x,y]$ such that $P(x,y)\mid Q(x^{2}, y^{2})$.

Let $\mathcal P(s, s', a, b, a', b') = Q(\frac{a^{2}}{s}, \frac{{a'}^{2}}{s'})\in\Qb[s,\frac1s,s',\frac1{s'},a,b,a',b']$. 
We see that 
\[\mathcal P(p^{k_1 + k_2 - 3}, p^{k_1'+k_2' - 3}, |a_p|, |b_p|, |a_p'|, |b_p'|) = Q(|\lambda_p|^{2}, |\lambda'_p|^{2}) = 0\]
for a positive proportion of $p$. A similar argument to the proof of \Cref{cor:normalised-trace} shows that $\mathcal{P}$ is non-zero and coprime to $s^\kappa-{s'}^{\kappa'}$. The result follows from \Cref{cor:absolute}.
\end{proof}

Taking $P(x, y) = x-y$, we deduce the following corollary, which generalises \cite{wong}*{Thm.~1.1} to the case of Siegel modular forms:

\begin{corollary}\label{cor:abs}
    Suppose that for a set of primes $p$ of positive upper density, we have $|\lambda_p| =  | \lambda_p'|$. Then $(k_1, k_2) = (k_1', k_2')$ and there is a Dirichlet character $\chi$ such that $\Pi\simeq\Pi'\otimes\chi$. 
\end{corollary}

We also deduce the following corollary, which strengthens \Cref{cor:power} and generalises \Cref{cor:arbitrary power} to the case of normalised eigenvalues.

\begin{corollary}
        Suppose that for a positive density of primes $p$, there exists a positive integer $n_p$ such that $\lambda_p^{n_p} = {\lambda_p'}^{n_p}$. Then $(k_1, k_2) = (k_1', k_2')$ and there is a Dirichlet character $\chi$ such that $\Pi\simeq\Pi'\tensor\chi$.
\end{corollary}

\begin{proof}
    If $\lambda_p^{n_p} = {\lambda_p'}^{n_p}$, then $|\lambda_p| = |\lambda_p'|$, and the result follows from \Cref{cor:abs}.
\end{proof}

\subsection{Distinguishing eigenforms by their Sato--Tate angles}

Assume that the central characters $\epsilon$ and $\epsilon'$ are trivial, so that the normalised eigenvalues $\lambda_p,\lambda_p'$ are real. Suppose further that the Ramanujan conjecture holds for $\pi$ and $\pi'$. This conjecture is known when the weights $k_2$ and $k'_2$ are at least $3$ \cite{Weissauer}*{Thm.~I}, but is open in the case that $k_2$ or $k_2'$ is equal to $2$. Under these assumptions, the normalised Hecke eigenvalues satisfy
\[\lambda_p, \lambda_p'\in[-4,4]\]
for each prime $p\notin S\cup S'$, so we can find $\theta_p,\theta_p'\in[0,\pi]$ such that 
\[4\cos\theta_p = \lambda_p\quad\text{and}\quad4\cos\theta'_p = \lambda'_p.\]

As an application of \Cref{cor:normalised-trace}, we show in the following result that $\theta_p$ and $\theta_p'$ satisfy a linear relation for a positive density of primes $p$ if and only if $\Pi$ and $\Pi'$ are twists of each other. This result generalises \cite{gun-murty-paul}*{Thm.~2}, which in turn builds on earlier work of Murty--Pujahari \cite{murty-pujahari}. Their result addresses the case of elliptic modular forms with $m = 1$ and $n = \pm 1$ and relies on the Sato--Tate conjecture. In contrast, our proof does not rely on the Sato--Tate conjecture, which is open in the case of Siegel modular forms.

\begin{corollary}\label{thm:angles}
    Let $m,n\in \Z\setminus\{0\}$ and $\alpha\in\R$. Assume that the set
    \[\set{p : m\theta_p + n\theta_p' = \alpha}\]
    has positive density. Then $(k_1, k_2) = (k_1', k_2')$ and there exists a Dirichlet character $\chi$ such that $\Pi\simeq\Pi'\tensor\chi$.
\end{corollary}

\begin{proof}
Suppose that $ m\theta_p + n\theta_p' = \alpha$. Expanding $\cos(m\theta + n\theta') =\cos\alpha$ using the addition and multiple angle formulas, we obtain a non-zero polynomial $P\in\Qb[c,c',s,s',a]$, monic in $a$, such that
    \[P(\cos\theta_p,\cos\theta_p',\sin\theta_p,\sin\theta_p',\cos\alpha) = 0.\]
    By \Cref{lem:divis}, we can construct a non-zero polynomial $Q\in \Qb[c,c',s,s',a]$, monic in $a$, such that
    \[Q(\cos\theta_p,\cos\theta_p',\sin^2\theta_p,\sin^2\theta_p',\cos\alpha) = 0.\]
    Set
    \[R(x, y, a) = Q\br{\frac x4, \frac y4, 1- \frac{x^2}{16},1-\frac{y^2}{16},a}.\]
    Since $Q$ is monic in $a$, so is $R$, so in particular $R\ne 0$. The construction of $R$ does not depend on $p$, and we have
    \[R(\lambda_p, \lambda_p', \cos\alpha) = 0\]
    whenever $m\theta_p + n\theta_p' = \alpha$. Since $\lambda_p, \lambda_p'$ are algebraic and $R$ is monic in $a$, it follows that $\cos\alpha$ is algebraic. Thus, if we define
    \[\mathcal{P}(x, y) = R(x, y, \cos\alpha)\in\Qb[x,y],\]
    then $\mathcal{P}(\lambda_p,\lambda'_p) = 0$ whenever $m\theta_p + n\theta_p' = \alpha$.

    It remains to show that $\mathcal{P}$ is non-zero. By construction,
    \begin{align*}
        \mathcal{P}(4\cos\theta,4\cos\theta') &= Q(\cos\theta,\cos\theta',\sin^2\theta,\sin^2\theta',\cos\alpha)\\
        &=\prod_{\delta,\delta'\in\{\pm1\}}P(\cos\theta,\cos\theta',\delta\sin\theta,\delta'\sin\theta',\cos\alpha).
    \end{align*}
    
    Let $\theta_0$ be such that $\cos(m\theta_0)\ne \cos\alpha$. 
    Since $P$ was obtained by expanding $\cos(m\theta + n\theta') = \cos\alpha$, evaluating this expression at $(\theta,\theta') =(\theta_0, 0)$, we see that
    \begin{align*}
        \mathcal P(4\cos\theta_0,4) &= P(\cos\theta_0, 1, \sin\theta_0, 0, \cos\alpha)^2P(\cos\theta_0, 1, -\sin\theta_0, 0, \cos\alpha)^2\\ &= (\cos\alpha-\cos(m\theta _0+ n\cdot 0))^2(\cos\alpha-\cos(-m\theta _0+ n\cdot 0))^2 \\
        &=(\cos\alpha-\cos(m\theta_0))^4\ne 0.
    \end{align*}

    Thus $\mathcal{P}\ne 0$ and the result follows from \Cref{cor:normalised-trace}.
\end{proof}

\subsection{Distinguishing eigenforms by their Satake parameters}\label{sec:satake}

For each prime $p\notin S\cup S'$, let $\{\beta_{1, p}, \beta_{2, p}, \beta_{3, p}, \beta_{4,p}\}$ be the roots of the Hecke polynomial
\[x^4- a_px^3 + \lb a_p^2 -a_{p^2} -\epsilon(p)p^{ k_1 + k_2 -4}\rb x^2 - a_p\epsilon(p)p^{ k_1 + k_2 -3}x + \epsilon(p^{2})p^{2 (k_1 + k_2-3)}\]
 corresponding to $\pi$, and for each $i$, let 
\[
\widetilde\beta_{i,p}=\frac{\beta_{i,p}}{p^{(k_1+k_2-3)/2}}.\]
Then the $\widetilde\beta_{i,p}$ are the Satake parameters of the transfer of the automorphic representation $\pi$ to $\GL_4(\AQ)$ after twisting to make the representation unitary. 
The symmetries of the Hecke polynomial ensure that the roots can be ordered so that
\[\beta_{1,p}\beta_{4,p} = \beta_{2,p}\beta_{3,p} = \epsilon(p)p^{k_1 + k_2-3}\]
and hence that
\[\widetilde\beta_{1,p}\widetilde\beta_{4,p} = \widetilde\beta_{2,p}\widetilde\beta_{3,p} = \epsilon(p).\]
In particular, the $\beta_{i,p}$ are completely determined by $\beta_{1, p}$ and $\beta_{2,p}$, and this choice is unique up to the symmetries $\beta_{1,p}\leftrightarrow\beta_{2,p}$ and $\beta _{i,p}\leftrightarrow\beta_{5-i,p} = \epsilon(p)p^{k_1 + k_2-3}/\beta_{i,p}$. Let $\beta_{i,p}'$ be the corresponding Satake parameters for $\pi'$.

\begin{theorem}\label{thm:satake}
    Let $P(s,s',x_1, x_2, x_1', x_2')\in \Qb[s^{\pm1},{s'}^{\pm1},x_1^{\pm1},x_2^{\pm1},{x_1'}^{\pm1}, {x_2'}^{\pm1}]$
    be a polynomial such that:
    \begin{itemize}
        \item $P$ is symmetric in $(x_1, x_2)$ and in $(x_1', x_2')$.
        \item $P$ is invariant under each of the substitutions $x_i\mapsto s/x_i$ and $x_i'\mapsto s'/x_i'$.
        \item $P$ is coprime to $s^{\kappa} - {s'}^{\kappa'}$.
    \end{itemize}
    Suppose that for a set of primes $p$ of positive upper density,
\[P(\epsilon(p)p^{k_1+k_2-3},\epsilon'(p)p^{k'_1+k'_2-3},\beta_{1,p}, \beta_{2,p},\beta_{1,p}',\beta_{2,p}')=0.\] 
Then $(k_1, k_2) = (k_1', k_2')$ and there is a Dirichlet character $\chi$ such that $\Pi\simeq\Pi'\otimes\chi$. 
\end{theorem}

\begin{proof}
    Set
\begin{align*}
   a &= x_1 + s/x_1 + x_2 + s/x_2,\\
   b &= x_1x_2/s + x_1/x_2 + x_2/x_1 + s/(x_1x_2) + 1,\\
   a' &= x'_1 + s'/x'_1 + x'_2 + s'/x'_2, \\
   b' &=x'_1x'_2/s' + x'_1/x'_2 + x'_2/x'_1 + s'/(x'_1x'_2) + 1.
\end{align*}
The symmetry assumptions on $P$ ensure that there is a polynomial $R\in\Qb[s^{\pm1},{s'}^{\pm1},a,b,a',b']$ such that
    \[P(s,s',x_1, x_2, x_1', x_2') = R(s,s',a,b,a',b').\]

Indeed, by the invariances under $x_i\mapsto s/x_i$ and $x_i'\mapsto s'/x_i'$, $P$ can be written as a polynomial in $\Qb[ s^{\pm1}, {s'}^{\pm1},u_1, u_2, u_1', u_2']$, where $u_i = x_i+s/x_i$ and $u_i' = x_i' + s'/x_i'$. The symmetries in $(x_1, x_2)$ and in $(x_1', x_2')$ then become symmetries in $(u_1, u_2)$ and $(u_1', u_2')$, so by the fundamental theorem of symmetric polynomials, $P$ can be written as a polynomial in $\Qb[ s^{\pm1}, {s'}^{\pm1},u_1+ u_2, u_1u_2, u_1' + u_2', u_1'u_2']$, or equivalently, as a polynomial $R$ in $\Qb[s^{\pm1},{s'}^{\pm1},a,b,a',b']$.
Moreover, since $P$ is coprime to $s^{\kappa}-{s'}^{\kappa'}$, so is $R$.

Let 
\[(s_p, s_p') =(\epsilon(p)p^{k_1+k_2-3},\epsilon'(p)p^{k'_1+k'_2-3}).\]
Observe that for every $p$, we have
\[a_p = \beta_{1,p} + \beta_{2,p} + s_p/\beta_{1,p} + s_p/\beta_{2,p},\]
\[b_p = \beta_{1,p}\beta_{2,p}/s_p + \beta_{1,p}/\beta_{2,p} + \beta_{2,p}/\beta_{1,p} + s_p/(\beta_{1,p}\beta_{2,p}) + 1,\]
and similarly for $a_p'$ and $b_p'$.
Hence,
\[P(s_p,s_p',\beta_{1,p},\beta_{2,p},\beta_{1,p}',\beta_{2,p}') = R(s_p,s_p',a_p,b_p,a_p',b_p')\]
and the result follows from \Cref{thm:b_p}.
\end{proof}

\begin{corollary}\label{cor:normalized satake}
  Assume that the central characters $\epsilon$ and $\epsilon'$ of $\pi$ and $\pi'$ are trivial. Let $P(x_1, x_2, x_1', x_2')\in \Qb[x_1^{\pm1},x_2^{\pm1},{x_1'}^{\pm1}, {x_2'}^{\pm1}]$  be a non-zero polynomial such that:
    \begin{itemize}
        \item $P$ is symmetric in $(x_1, x_2)$ and in $(x_1', x_2')$.
        \item $P$ is invariant under each of the substitutions $x_i\mapsto 1/x_i$ and $x_i'\mapsto 1/x_i'$.
    \end{itemize} 
    Suppose that for a set of primes $p$ of positive upper density,
\[P(\widetilde\beta_{1,p}, \widetilde\beta_{2,p},\widetilde\beta_{1,p}',\widetilde\beta_{2,p}')=0.\] 
Then $(k_1, k_2) = (k_1', k_2')$ and there is a Dirichlet character $\chi$ such that $\Pi\simeq\Pi'\otimes\chi$. 
\end{corollary}

\begin{remark}
    It is possible to formulate a statement of \Cref{cor:normalized satake} that applies with arbitrary central characters. However, the statement will be messier, due to the fact that $\widetilde\beta_{1,p}\widetilde\beta_{4,p} = \epsilon(p)$, which is not identically $1$.
\end{remark}

\begin{proof}
  By \Cref{lem:divis}, there exists a polynomial $Q\in \Qb[x_1^{\pm1},x_2^{\pm1},{x_1'}^{\pm1}, {x_2'}^{\pm1}]$ such that    
\begin{equation*}
P(x_1, x_2, x_1', x_2')\mid Q(x_1^{2},x_2^{2},{x_1'}^{2},{x_2'}^{2}).
\end{equation*}
The construction of $Q$ ensures that $Q$ satisfies the symmetry hypotheses. Define
$\mathcal P(s,s',x_1, x_2, x_1', x_2')\in \Qb[s^{\pm1},{s'}^{\pm1},x_1^{\pm1},x_2^{\pm1},{x_1'}^{\pm1}, {x_2'}^{\pm1}]$ by
\[
\mathcal P(s,s',x_1, x_2, x_1', x_2')= Q(x_1^{2}/s,x_2^{2}/s, {x_1'}^{2}/{s'},{x_2'}^{2}/{s'}).
\]

By construction, $\mathcal P$ satisfies all the symmetry and invariance conditions of \Cref{thm:satake}. Moreover, $\mathcal{P}$ is coprime to $s^{\kappa} - {s'}^{\kappa'}$. Indeed, if $\mathcal{P}$ had a common factor with $s^\kappa-s'^{\kappa'}$, then $\mathcal P(s^2,{s'}^2, x_1, x_2, x_1',x_2')$ would have a common factor with $s^{2\kappa}-s'^{2\kappa'}$. But 
\[\mathcal P(s^2,{s'}^2, x_1, x_2, x_1',x_2') 
 = Q((x_1/s)^{2}, (x_2/s)^{2}, (x_1'/s')^{2}, (x_2'/s')^{2})\]
 and by the construction in \Cref{lem:divis}, the right-hand side is a product of the polynomials $P(\zeta_1 x_1/s, \zeta_2 x_2/s, \zeta_1' x_1'/s', \zeta_2' x_2'/s')$, for $\zeta_i,\zeta_i'\in \{\pm1\}$. The change of variables $x_i\mapsto sx_i$ and $x_i'\mapsto s'x_i'$ shows that each of these polynomials is coprime to $s^{2\kappa}-{s'}^{2\kappa'}$.

For each prime $p$, we have
\[\widetilde\beta_{i,p}^{2} = \frac{\beta_{i,p}^{2}}{s_p}\quad\text{ and }\quad (\widetilde\beta'_{i,p})^{2} = \frac{(\beta'_{i,p})^{2}}{{s'_p}},\]
where
\[(s_p, s_p') =(p^{k_1+k_2-3},p^{k'_1+k'_2-3}) .\] 
Hence,
    \[\mathcal{P}(s_p, s_p',\beta_{1,p},\beta_{2,p},\beta_{1,p}',\beta_{2,p}')=Q(\widetilde\beta_{1,p}^{2}, \widetilde\beta_{2,p}^{2},\widetilde\beta_{1,p}'^{2},\widetilde\beta_{2,p}'^{2})=0\]
   for a positive proportion of primes. The result follows from \Cref{thm:satake}.
\end{proof}

We immediately deduce the following corollary, which strengthens \cite{kumar-meher-shankhadar}*{Thm.~3.1} to higher level forms and generalises \cite{weiss2018image}*{Cor.~5.11}.

\begin{corollary}
    Suppose that $\Pi$ is not a character twist of $\Pi'$. Then for a set of primes of density $1$:
    \begin{enumerate}
        \item The Satake parameters $\beta_{i,p}$ and $\beta_{i,p}'$ are pairwise distinct.
        \item If $\epsilon$ and $\epsilon'$ are trivial, then the normalised Satake parameters $\widetilde\beta_{i,p}$ and $\widetilde\beta_{i,p}'$ are pairwise distinct.
    \end{enumerate}
\end{corollary}

\begin{proof}

\begin{enumerate}
    \item Take $P$ to be the square of the product of all the differences in the set
    \[\{x_1, x_2, s/x_1, s/x_2, x_1', x_2', s'/x_1', s'/x_2'\}\]
    in \Cref{thm:satake}.
    \item Take $P$ to be the square of the product of all the differences in the set
    \[\{x_1, x_2, 1/x_1, 1/x_2, x_1', x_2', 1/x_1', 1/x_2'\}\]
    in \Cref{cor:normalized satake}.
\end{enumerate}
\end{proof}

Recall from \Cref{sec:eigenvalues} that $\lambda_{p^2}$ and $\lambda_{p^2}'$ are the normalised Hecke eigenvalues of $f$ and $f'$ under the operator $T(p^2)$. Using \Cref{cor:normalized satake} we show that $f$ is determined up to twist by $\lambda_{p^2}$ for a positive density of $p$. 

    \begin{corollary}\label{cor:square}
   Suppose that $f$ and $f'$ are Siegel paramodular newforms. Suppose that $\lambda_{p^2}=\lambda_{p^2}'$ for a set of primes $p$ of positive upper density. 
Then $(k_1, k_2) = (k_1', k_2')$ and there is a quadratic Dirichlet character $\chi$ such that $f$ is a scalar multiple of $f'\otimes\chi$. 
    \end{corollary}
    \begin{proof}
  From \Cref{eq:l-function}, we have
        \[
            \lambda_{p^2}= \sum_{1\le i \le 4} \widetilde\beta_{i,p}^2+ \sum_{1\le i<j\le 4}\widetilde\beta_{i,p}\widetilde\beta_{j,p}- \frac1p \quad {\rm and} \quad 
             \lambda_{p^2}'= \sum_{1\le i \le 4} \widetilde\beta_{i,p}'^2+ \sum_{1\le i<j\le 4}\widetilde\beta_{i,p}'\widetilde\beta_{j,p}'- \frac1p.
      \]
The result follows from \Cref{cor:normalized satake} by taking
     \[ P(x_1,x_2,x_1',x_2'):=
     \br{\sum_{\substack{i=1,2\\ \delta\in\{\pm1\}}}x_i^{2\delta} +   \sum_{\delta_1,\delta_2\in\{\pm1\}} x_1^{\delta_1}x_2^{\delta_2}}   -\br{\sum_{\substack{i=1,2\\ \delta\in\{\pm1\}}} {x_i'}^{2\delta} + \sum_{\substack{\delta_1,\delta_2\in\{\pm1\}}} {x_1'}^{\delta_1}{x_2'}^{\delta_2}}.
       \]
    \end{proof}

\subsection{Distinguishing eigenforms by their $L$-functions}

If $r\:\Gf\to \GL_n$ is an algebraic representation of $\Gf$, then we can define an $L$-function $L(\pi, s, r)$
\[L(\pi, s, r) = \prod_p L_p(\pi, s,r),\]
where, if $p\notin S$, then
\[L_p(\pi, s, r) = \det(1 - r(\diag(\beta_{1, p}, \beta_{2, p}, \beta_{3, p}, \beta_{4, p}))p^{-s})\ii.\]
Here, $\beta_{1, p}, \beta_{2, p}, \beta_{3, p}, \beta_{4, p}$ are the Satake parameters defined in the previous section. In particular, if we take $r\:\Gf\to \GL_4$ to be the inclusion map, then we have
\[L(\pi, s, r) = L_{\mathrm{spin}}(\pi, s),\]
while if we take $r = \std\:\Gf\to\PGSp_4\xrightarrow{\sim}\SO_5\hookrightarrow\GL_5$, then
\[L(\pi, s, r) = L_{\std}(\pi, s).\]

Expanding the Euler product, we can write
\[L(\pi, s, r) = \sum_{n\ge 1}r_nn^{-s},\]
 where, for a prime $p\notin S$, $r_p$ is the coefficient of $p^{-s}$ in the Euler factor $L_p(\pi, s, r)$. We can similarly define $L(\pi', s, r)$ and $r_p'$. 
 
 Thus, if $r\:\Gf\to \GL_4$ is the inclusion map, $r_p = a_p$, while if $r = \std$, then $r_p = b_p$. We can also take $r = \Sym^m\:\Gf\to \GL_{m+3\choose m}$ to recover the $m$-th symmetric-power $L$-function, $r = \mathrm{adj}\:\Gf\to \GL_{10}$ for the adjoint $L$-function, or $r$ to be any other representation of arithmetic interest.

 Note that if we take $r\:\Gf\to\GL_1$ to be the similitude character, then $L(\pi, s, r) = L(\epsilon, s - (k_1 + k_2 - 3))$, which depends only on the weight and character of $\pi$. Excluding this case, we show that if $r$ is not a direct sum of powers of the similitude character, then $\Pi$ is determined up to twist by $r_p$:
 
\begin{theorem}\label{thm:l-function}
    Let $r\:\Gf\to\GL_n$ be any semisimple algebraic representation of $\Gf$. Assume that $r$ is not isomorphic to a direct sum of one-dimensional representations. 
    Let $P(x,y)\in\Qb[x,y]$ be any non-zero polynomial, and suppose that for a set of primes $p$ of positive density
    \[P(r_p, r_p') = 0.\]
Then $(k_1, k_2) = (k_1', k_2')$ and there is a Dirichlet character $\chi$ such that $\Pi\simeq\Pi'\tensor\chi$. 
\end{theorem}

\begin{proof}
    Note that we have $\beta_{1,p}\beta_{4,p} = \beta_{2,p}\beta_{3,p} = \epsilon(p)p^{k_1 + k_2 -3}$, and similarly for the $\beta_{i,p}'$. Let 
    \[R(s, x_1, x_2) = \Tr r(\diag(x_1, x_2, s/x_2, s/x_1)),\]
    so that $R(\epsilon(p)p^{k_1 + k_2 -3}, \beta_{1, p}, \beta_{2,p}) = r_p$. Let
    \[Q(s, s', x_1,  x_2,x_1', x_2') = P(R(s, x_1, x_2) ,R(s', x_1', x_2')).\]
    The conjugation invariance of the trace map ensures that $Q$ satisfies the first two hypotheses of \Cref{thm:satake}. 
    
    Suppose for contradiction that $Q$ is not coprime to $s^{\kappa} - {s'}^{\kappa'}$. Then there exist roots of unity $\zeta, \zeta'$ and integers $n, n'$ such that $Q(\zeta t^n, \zeta' t^{n'}, x_1,  x_2,x_1', x_2')$ is the zero polynomial of the polynomial ring $\Qb[t^{\pm1}, x_1^{\pm1},x_2^{\pm1}, {x_1'}^{\pm1}, {x_2'}^{\pm1}]$. Now, since $r$ is not a direct sum of one-dimensional representations, it does not factor through the similitude character $\simil\:\Gf\to \GL_1$. It follows that for some fixed $t_0\in\Qb\t$, the two polynomials $T(x_1, x_2) = R(\zeta t_0^n, x_1, x_2)$ and $T'(x_1',x_2') = R(\zeta't_0^{n'},x_1',x_2')$ are non-constant elements of $\Qb[x_1^{\pm 1},x_2^{\pm1}]$ and $\Qb[{x_1'}^{\pm 1},{x_2'}^{\pm1}]$ respectively.
    
    By assumption, $P({R}(\zeta t^n, x_1, x_2),{R}(\zeta' t^{n'}, x_1', x_2')) = 0$ and hence, 
    \[P(T( x_1, x_2),T'(x_1', x_2'))\]
    is the zero polynomial. Since $T$ and $T'$ are non-constant, $T( x_1, x_2)$ and $T'( x_1', x_2')$ are algebraically independent, so $P$ must be the zero polynomial, a contradiction.

    Thus $Q$ satisfies the hypotheses of \Cref{thm:satake}, and the result follows.
\end{proof}

\begin{remark}\label{rmk:symmetric}
    Suppose that $f$ and $f'$ are elliptic modular forms. Then, taking $r = \Sym^m\:\GL_2\to \GL_{m+1}$, we have $r_p = a_{p^m}$. Thus, from the modular forms analogue of \Cref{thm:l-function}, we deduce that for a non-zero polynomial $P\in\Qb[x,y]$, if $P(a_{p^m},a_{p^m}')=0$ for a positive density of primes $p$, then $f$ is a twist of $f'$. This result strengthens \Cref{cor:apn}. A similar result should hold for Siegel modular forms too, however, the proof will be more convoluted. Indeed, if we take $r = \Sym^m\:\Gf\to \GL_{{m+3}\choose m}$,  then $r_p$ is the $p^m$-th coefficient of $L_{\mathrm{spin}}(\pi, s) = L(\epsilon, 2s - u + 1)\sum_{n\ge 1} a_nn^{-s}$, which is not equal to $a_{p^m}$ unless $m = 1$. 
\end{remark}

\subsection{Distinguishing eigenforms by the distances of their eigenvalues}

Suppose that $\pi$ and $\pi'$ both have integer Hecke eigenvalues. We deduce the following result:

\begin{corollary}\label{cor:distance}
    Suppose that both $\pi$ and $\pi'$ have integer Hecke eigenvalues and that there does not exist a Dirichlet character $\chi$ such that $\Pi\cong\Pi'\tensor\chi$. Let $X$ be any real number. Then
    \[\set{p : |a_p - a'_p|<X}\]
    has density $0$.
\end{corollary}

\begin{proof}
    Apply \Cref{cor:trace} to the polynomial
    \[\prod_{i\in[-X,X]\cap \Z}(x-y + i).\]
\end{proof}

More generally, we can always find a number field $E$ such that $\pi$ and $\pi'$ both have Hecke eigenvalues in its ring of integers $\O_E$. Following \cite{Silverman}*{VIII.5}, for any $x\in E$, we define a height function
\[H_E(x) = H_E([x : 1]) = \prod_{v\in M_E}\max{(|x|_v^{[E_v : \Q_{v}]},1)},\]
where $M_E$ denotes the set of places of $E$, and for $v\in M_E$, $|\cdot|_v$ is the corresponding absolute value. 
The function $H_E(x)$ has the property that for any $X\iR$, the set
\[\set{x\in \O_E : H_E(x) < X}\]
is finite (see \cite{Silverman}*{VIII.5 Thm.~5.11}). 

Note that if $E = \Q$ and $x=a/b\iQ$ with $\gcd(a, b) = 1$, then $H_E(x) = \max(|a|,|b|)$ is the usual height function \cite{Silverman}*{VIII.5 Rem.~5.5}. Thus, the following corollary generalises \Cref{cor:distance}:

\begin{corollary}
Suppose that both $\pi$ and $\pi'$ have Hecke eigenvalues in the ring of integers $\O_E$ of a number field $E$, and let $X$ be any real number. Assume that the set
  \[\set{p : H_E(a_p - a_p') < X}\]
has positive upper density. Then $(k_1, k_2) = (k_1', k_2')$ and there is a Dirichlet character $\chi$ such that $\Pi\simeq\Pi'\otimes\chi$.
\end{corollary}
\begin{proof}
    Apply \Cref{cor:trace} to the polynomial
    \[\prod_{\substack{i\in\O_E \\ H_E(i)<X}}(x-y + i).\]
\end{proof}

\section{Images of Galois representations}

\subsection{The group $\Gf$}

For a ring $R$, let
        \[\Gf(R) = \set{\gamma\in\GL_4(R): \gamma^t J \gamma = \nu J,\ \nu \in R\t},\]
        where  $J = \br{\begin{smallmatrix}
        	0 & 0 & 0 & 1\\
        	0 & 0 & 1 & 0\\
        	0 & -1 & 0 & 0\\
        	-1 & 0 & 0 & 0
        \end{smallmatrix}}$. For $\gamma\in \Gf(R)$, we call the constant $\nu$ the \emph{similitude} of $\gamma$ and denote it $\simil(\gamma)$. Let $\Sp_4$ denote the subgroup of $\Gf$  of elements with $\simil(\gamma)= 1$. We view $\Gf$ as an affine algebraic group over $\Spec\Qb$. Explicitly, we have
        \[\Gf = \Spec\O(\Gf),\]
        where
        \begin{equation*}
            \O(\Gf) = \frac{\Qb[\set{a_{ij} : 1\le i,j\le 4}, t]}{\left\langle \sum_{j,k}ta_{ji}J_{jk}a_{kl} - J_{il} : 1 \le i < l \le 4\right\rangle}.
        \end{equation*}

\subsection{Galois representations attached to Siegel modular forms}

Let $\pi$ and $\pi'$ be as in the introduction. In particular, we assume throughout that $\pi$ and $\pi'$ are of general type and that neither has CM, RM, nor arises as a symmetric cube lift.

Let $E$ denote the number field generated by the values of the central characters $\epsilon, \epsilon'$ and the Hecke eigenvalues of $\pi$ and $\pi'$. We fix once and for all an embedding $E\hookrightarrow \Qb$ and embeddings $\Qb\hookrightarrow\Qlb$ for each prime $\l$.

By the work of Taylor, Laumon and Weissauer \cites{taylor1993, Laumon, Weissauer, Weissauersymplectic} when $k_2> 2$, and Taylor \cite{taylor1991galois} when $k_2=2$ (see also \cite{mok2014galois}), we can associate a compatible system of Galois representations to $\pi$. 
For each prime $\l$, there exists a continuous semisimple symplectic Galois representation
	\[\rho_\l\:\Ga\Q\to \Gf(\Qlb)\]
	that is unramified at all primes $p\notin \{\l\}\cup S$, where $S$ is the finite set of primes where $\pi$ is ramified. The Galois representation is characterised by the equalities
	\[\Tr\rho_{\l}(\Frob_p)= a_p,\qquad \simil\rho_{\l}(\Frob_p) = \epsilon(p)p^{k_1 + k_2 -3},\]
	for all primes $p\notin \{\l\}\cup S$. Let $\rho_\l'$ be the corresponding Galois representation attached to $\pi'$.

\begin{remark}
    We have opted to work with $\l$-adic representations $\rho_\l\:\Ga\Q\to\Gf(\Qlb)$ taking values in $\Qlb$ rather than $\lambda$-adic representations $\rho_\lambda\:\Ga\Q\to\Gf(\elb)$ taking values in $\elb$. This choice leads to simpler notation, and suffices for our purposes.
\end{remark}

	By work of Ramakrishnan \cite{Ramakrishnan}, Dieulefait--Zenteno \cite{DZ} and the second author \cite{weissthesis,weiss2018image}, the image of $\rho_\l$ is generically large, in the following sense. Let $\LL_\pi$ be the set of rational primes $\l\ge 5$ such that $\rho_\l|_{\Ql}$ is de Rham if $\l\in S$ and crystalline if $\l\notin S$. Then $\LL_\pi$ is just the set of primes $\l\ge 5$ if $k_2>2$, while if $k_2 = 2$, then $\LL_\pi$ has Dirichlet density $1$ \cite{weiss2018image}*{Thm.\ 1.1}.
	
	\begin{theorem}\label{thm:large-image}
        \begin{enumerate}[leftmargin=*]
			\item If $\l\in \LL_\pi$, then $\rho_\l$ is absolutely irreducible.
			\item For all but finitely many $\l\in\LL_\pi$, the image of $\rho_\l$ contains a subgroup conjugate to $\Sp_4(\Zl)$.
		\end{enumerate}
	\end{theorem}

    \begin{proof}
        Part $(i)$ is exactly \cite{weiss2018image}*{Thm.~1.1}. Part $(ii)$ follows from \cite{weiss2018image}*{Thm.~1.2} and the fact that if $\l\ge 5$ and $X$ is a closed subgroup of $\Sp_4(\Zl)$ that surjects onto $\Sp_4(\Fl)$, then $X = \Sp_4(\Zl)$ (cf.\ \cite{serre-abelian}*{Lemma 3.4.3}).
    \end{proof}

Now, let $G_\l$ denote the Zariski closure of  $\im(\rho_\l)$ in $\GSp_{4/\Qlb}$. Then $G_\l$ is a Zariski closed subgroup of $\Gf(\Qlb)$ that contains a subgroup that is conjugate to $\Sp_4(\Zl)$ and whose similitude has infinite image. We immediately deduce the following corollary:

\begin{corollary}\label{cor:zariski-closed}
    For all but finitely many $\l\in \LL_\pi$, we have $G_\l = \Gf(\Qlb)$.
\end{corollary}

We define $\LL_{\pi'}$ in the analogous way, and set $\LL$ to be the set of primes obtained by deleting from $\LL_\pi\cap\LL_{\pi'}$ the finitely many exceptional primes in \Cref{cor:zariski-closed}.

\subsection{The algebraic group $\Gkk$}\label{sec:gkk}

Recall that 
 \[\Gkk = \set{(\gamma, \gamma')\in \Gf\times\Gf : \simil(\gamma)^{\kappa} = \simil(\gamma')^{\kappa'}},\]
where $\kappa$ and $\kappa'$ are the integers defined in \eqref{def:kappa}.
Let $d = \gcd(\kappa,\kappa')$, and for each root of unity $\zeta\in\mu_d$, let 
\[\Gkk^\zeta = \set{(\gamma, \gamma')\in \Gf\times\Gf : \simil(\gamma)^{\kappa/d} = \zeta\simil(\gamma')^{\kappa'/d}}. \]
Then there is a coset decomposition
 \[\Gkk = \bigsqcup_{\zeta\in\mu_d}\Gkk^\zeta.\]
In this subsection, we prove that $\Gkk^1$ is connected, so that this is the decomposition of $\Gkk$ into its connected components.

As an affine algebraic group over $\Qb$, $\Gf\times\Gf = \Spec\O(\Gf\times\Gf)$, where
\[\O(\Gf\times\Gf) = \frac{\Qb\left[\set{a_{ij}, a'_{ij} : 1\le i,j\le 4}, t, t'\right]}{\left\langle \sum_{j,k}ta_{ji}J_{jk}a_{kl} - J_{il}, \sum_{j,k}t'a_{ji}'J_{jk}a_{kl}' - J_{il} : 1 \le i < l \le 4\right\rangle}. \]

\begin{proposition}\label{prop:ufd}
    Let $G = \Gf\times\Gf$, viewed as an algebraic group over $\Qb$. Then $\O(G)$ is a unique factorisation domain. 
\end{proposition}

\begin{proof}
    Since $G$ is a non-singular algebraic variety, we have $\Pic(G) = \Cl(G)$ \cite{hartshorne}*{Cor.~II.6.16}. Since $\Cl(G) = 0$ if and only if $\O(G)$ is a UFD, it is sufficient to show that $\Pic(G) = 0$.

    We adapt the argument of \cite{youcis}. By \cite{iversen}*{Prop.~2.6}, there is an exact sequence
    \[0\to X^*(\PGSp_4\times\PGSp_4)\to X^*(G) \to X^*(\GG_m^2)\to \Pic(\PGSp_4\times\PGSp_4)\to\Pic(G) \to 0, \]
    where $X^*(G) = \Hom(G, \GG_m)$. Since $\PGSp_4$ is semisimple, we have $X^*(\PGSp_4\times\PGSp_4) = 0$. 
    Moreover, $X^*(\GG_m^2)\simeq\Z^2$ and
    \[X^*(\Gf) = X^*(\Gf/\Sp_4) = X^*(\GG_m) \simeq\Z,\]
    so $X^*(G) \simeq\Z^2$. Thus, the exact sequence becomes
    \[0\to0\to \Z^2\to \Z^2\to \Pic(\PGSp_4\times\PGSp_4)\to \Pic(G)\to 0.\]
    Here, the map $\Z^2\to\Z^2$ is multiplication by $2$: any character $G\to \GG_m$ factors through the similitude, and the composition $\GG_m\to \Gf\xrightarrow{\simil}\GG_m$ is the map $x\mapsto x^2$.
    
    By \cite{iversen}*{Prop.~3.5}, we have $\Pic(\PGSp_4\times\PGSp_4) = X^*(\pi_1(\PGSp_4\times\PGSp_4))$.  Since the map $\Sp_4\to \PGSp_4$ is a central isogeny with kernel $\mu_2$ and since $\Sp_4$ is simply connected, it 
 follows that $\pi_1(\PGSp_4)\simeq\mu_2$. Thus,
    \[X^*(\pi_1(\PGSp_4\times\PGSp_4)) = X^*(\mu_2\times\mu_2) = (\Z/2\Z)^2.\]
    It follows from the exact sequence that $\Pic(G) = 0$.
\end{proof}

\begin{corollary}\label{prop:comp-connected}
    Suppose that $\gcd(\kappa, \kappa') = 1$. Then $\Gkk$ is connected.
\end{corollary}

\begin{proof}
    In general, we have
    \[\O(\Gkk) = \frac{\O(\Gf\times\Gf)}{\langle t^{\kappa} - {t'}^{\kappa'}\rangle}.\]

    Since $\gcd(\kappa, \kappa') = 1$, the polynomial $t^{\kappa} - {t'}^{\kappa'}$ is irreducible in $\O(\Gf\times\Gf)$.
    
    By \Cref{prop:ufd}, it follows that $t^{\kappa} - {t'}^{\kappa'}$ is prime, and thus that $\O(\Gkk)$ is an integral domain. So $\Gkk$ is irreducible, and in particular, connected.
\end{proof}

\begin{corollary}\label{prop:connected-comps}
    The decomposition
    \[\Gkk = \bigsqcup_{\zeta\in\mu_d}\Gkk^\zeta,\]
    decomposes $\Gkk$ into connected components, with identity connected component $\Gkk^1 = \G_{\kappa/d,\kappa'/d}$. 
\end{corollary}

\subsection{Joint large image and the proof of \Cref{thm:joint-large-image}}

For each prime $\l$, let
\[R_\l = \rho_\l\times\rho_\l'\:\Ga\Q\to \Gkk(\Qlb)\sub\Gf(\Qlb)\times\Gf(\Qlb)\]
be the product Galois representation associated to $\pi$ and $\pi'$. Recall that $d = \gcd(\kappa, \kappa')$.

\begin{lemma}\label{lem:all-components}
    Let $\zeta\in\mu_d$ and let $\Gkk^{\zeta}$
    be the corresponding connected component of $\Gkk$. Then $\im(R_\l)\cap\Gkk^{\zeta}(\Qlb)\ne \emptyset$.
\end{lemma}

\begin{proof}
    Let $p$ be a prime and assume that $p\notin S\cup S'\cup\{\l\}$. Since
    \[(\epsilon(p)p^{k_1 + k_2 -3})^{\kappa} = (\epsilon'(p)p^{k_1' + k_2' -3})^{\kappa'},\]
    it follows that
    \[\epsilon(p)^{\kappa} = \epsilon'(p)^{\kappa'}\]
    and
     \[(p^{k_1 + k_2 -3})^{\kappa} = (p^{k_1' + k_2' -3})^{\kappa'}.\]
    Therefore, we have
         \[(p^{k_1 + k_2 -3})^{\kappa/d} = (p^{k_1' + k_2' -3})^{\kappa'/d}.\]
    By the minimality of $\kappa$ and $\kappa'$, the character $\epsilon^{\kappa/d}(\epsilon')^{-\kappa'/d}$ is a Dirichlet character of order $d$. Hence, there exists a prime $p\notin S\cup S'\cup\{\l\}$ such that $\epsilon(p)^{\kappa/d} = \zeta\epsilon'(p)^{\kappa'/d}$. Then \[(\simil\rho_\l(\Frob_p))^{\kappa/d} = \epsilon(p)^{\kappa/d}(p^{k_1 + k_2 - 3})^{\kappa/d}= \zeta\epsilon'(p)^{\kappa'/d}(p^{k_1' + k_2' -3})^{\kappa'/d} = \zeta(\simil\rho_\l'(\Frob_p))^{\kappa'/d}.\]
    Thus, $R_\l(\Frob_p)\in\Gkk^\zeta(\Qlb)$.
\end{proof}

\begin{proof}[Proof of \Cref{thm:joint-large-image}]
First note that if $\Pi\simeq\Pi'\tensor\chi$, then $\rho_\l\simeq\rho_\l'\tensor\chi$, and $\Gamma_\l$ contains a finite index subgroup isomorphic to the diagonal embedding of $\Gf(\Qlb)$ into $\Gkk(\Qlb)$. Thus $\Gamma_\l\subsetneq\Gkk(\Qlb)$.

For the converse, fix a prime $\l\in\LL$. By \Cref{lem:all-components}, we see that $\Gamma_\l = \Gkk(\Qlb)$ if and only if their identity connected components $\Gamma_\l^\circ$ and $\Gkk^1(\Qlb)$ are equal.

    By \Cref{cor:zariski-closed}, each of the two projections of $\Gamma_\l^\circ$ onto $\Gf(\Qlb)$ is surjective. Hence, by Goursat's Lemma, there exist normal subgroups $N, N'$ of $\Gf(\Qlb)$ such that $\Gamma_\l^\circ$ is the graph of an isomorphism $\Gf(\Qlb)/N\cong\Gf(\Qlb)/N'$. 
    
    First suppose that $\Sp_4(\Qlb) \sub N$. Then we must also have $\Sp_4(\Qlb)\sub N'$, and the two quotients $\Gf(\Qlb)/N$ and $\Gf(\Qlb)/N'$ are contained in $\Qlb\t$. It follows that $\Gamma_\l^\circ$ is a fibre product of $\Gf(\Qlb)$ with itself over some powers of the similitude map. Thus, $\Gamma_\l^\circ = \G_{a,a'}$ for some integers $a, a'$. Since $\Gamma_\l^\circ$ is connected and contained in the connected group $\Gkk^1(\Qlb) = \G_{\kappa/d,\kappa'/d}(\Qlb)$, it follows that $\Gamma_\l^\circ = \Gkk^1(\Qlb)$. Hence, $\Gamma_\l = \Gkk(\Qlb)$.

    So suppose that $\Sp_4(\Qlb) \not\sub N$. Then we must also have $\Sp_4(\Qlb) \not\sub N'$, so $N$ and $N'$ are contained in the subgroup of scalar matrices in $\Gf(\Qlb)$. It follows that the projective image of $\Gamma^\circ_\l$ in $\PGSp_4(\Qlb)\times\PGSp_4(\Qlb)$ is isomorphic to the diagonal embedding of $\PGSp_4(\Qlb)$.
    
    We deduce that the projective Galois representations
    \[\Proj\rho_\l\:\Ga\Q\to\PGSp_4(\Qlb)\]
    and
    \[\Proj\rho_\l'\:\Ga\Q\to\PGSp_4(\Qlb)\]
    are isomorphic. It follows that $\rho_\l\simeq\rho_\l'\tensor\chi$ for some character $\chi$. Now, since $\l\in\LL$, it follows that $\rho_\l$ and $\rho_\l'$ are Hodge--Tate with Hodge--Tate weights $\{0, k_2-2, k_1-1, k_1+k_2-3\}$ and $\{0, k'_2-2, k'_1-1, k'_1+k'_2-3\}$ \cite{mok2014galois}*{Thm.~3.1, Thm.~4.11}. Comparing Hodge--Tate weights, it follows that $\chi$ must be Hodge--Tate with Hodge--Tate weight $0$, i.e.\ $\chi$ must correspond to a Dirichlet character. 

    Since $\rho_\l\simeq\rho_\l'\tensor\chi$, by the strong multiplicity one theorem for $\GL_4(\AQ)$, we have $\Pi\cong\Pi'\tensor\chi$.
\end{proof}

\begin{remark}\label{rem:quadratic}
    Suppose that $\pi$ and $\pi'$ both have trivial central character. Then since $\simil\rho_\l \simeq \simil(\rho'_\l\tensor\chi) = \chi^2\simil\rho_\l'$, it follows that $\chi$ must be a quadratic character.
\end{remark}

\section{Proofs of \Cref{thm:b_p,thm:galois}}

\subsection{Proof of \Cref{thm:galois}}

We first recall the following algebraic version of the Chebotarev density theorem, due to Rajan.
	\begin{theorem}\label{rajan_result}\cite[Theorem 3]{raj}
		Let $E$ be a finite extension of $\Q_\ell$ and let $\mathcal G$ be
		an algebraic group over $E$. Let $\mathcal X$ be a closed subscheme of $\mathcal G$ defined over $E$ that is stable
		under the adjoint action of $\mathcal G$. Suppose that 
		$$
		R : \Ga\Q \rightarrow \mathcal G(E)
		$$
		is a Galois representation that is unramified outside a finite set of primes.
		Let $\Gamma$ denote the Zariski closure of  $\im(R)$ in $\mathcal G(E)$, with identity connected component $\Gamma^\circ$ and component group $\Phi = \Gamma/\Gamma^\circ$. For each $\phi \in \Phi$, let $\Gamma^\phi$ denote the	corresponding connected component. Let
	\[
		\Psi=\{\phi \in \Phi:\Gamma^\phi \subset \mathcal X\}.\]
		Then the set of primes $p$ such that $R(\Frob_p) \in \mathcal X(E) \cap \im(R)$ has density $\frac{|\Psi|}{|\Phi|}$.
	\end{theorem}

    \begin{proof}[Proof of \Cref{thm:galois}]
        Fix a prime $\l\in\LL$, where $\LL$ is as in \Cref{thm:joint-large-image}. Suppose that $\Pi$ is not a character twist of $\Pi'$. 
        Then by \Cref{thm:joint-large-image} the Zariski closure $\Gamma_\l$ of $R_\l$ in $(\Gf\times\Gf)_{/\Qlb}$ is $\Gkk(\Qlb)$. 
        Hence, there exists a finite Galois extension $E/\Ql$ such that $R_\l$ takes values in $\Gkk(E)$ and such that the Zariski closure of $R_\l$ in ${\Gkk}_{/E}$ is $\Gkk(E)$. Taking $E$ large enough, we can assume that $\phi$ is defined over $E$ and that the decomposition of \Cref{prop:connected-comps} occurs over $E$.

        Now, let $\mathcal{X}$ be the vanishing set of $\phi$. Since $\phi\in\O(\Gkk)^{\Gkk}$, $\mathcal{X}$ is stable under the adjoint action of $\Gkk$. Since $\phi$ is not identically zero on any connected component of $\Gkk$, it follows that no connected component of $\Gkk$ is contained in $\mathcal{X}$.
        Hence, by \Cref{rajan_result}, the set
        \[\set{p : R_\l(\Frob_p)\in\mathcal{X}(E)}\]
        has density $0$.    
        But, by definition
        \[\set{p : R_\l(\Frob_p)\in\mathcal{X}(E)} = \set{p : \phi(R_\l(\Frob_p))=0}.\]
        It follows that this set must have density $0$.
    \end{proof}
    
    \subsection{Proof of \Cref{thm:b_p}}

    We deduce \Cref{thm:b_p} from \Cref{thm:galois}.
    \begin{lemma}\label{lem:non-vanishing}
     Let
		$P(s,s',a,b,a',b')\in  \Qb[s, \frac1s, s',\frac1{s'},a, b,a',b']$. Assume that $P$ is coprime to $s^{\kappa} - {s'}^{\kappa'}$. Let $\phi$ be the element
    \[\phi\:(\gamma,\gamma')\mapsto P(\simil(\gamma),\simil(\gamma'),\tr(\gamma), \tr\std(\gamma), \tr(\gamma'), \tr\std(\gamma'))\]
    of $\O(\Gkk)^{\Gkk}$. Then $\phi$ does not vanish on any connected component of $\Gkk$.
\end{lemma}

\begin{proof}
    By \Cref{prop:connected-comps}, any connected component of $\Gkk$ is of the form
    \[\Gkk^\zeta = \set{(\gamma, \gamma')\in \Gf\times\Gf : \simil(\gamma)^{\kappa/d} = \zeta\simil(\gamma')^{\kappa'/d}},\]
    where $d = \gcd(\kappa, \kappa')$ and $\zeta\in\mu_d$. Let $\zeta'\in\Qb$ be such that $(\zeta')^{\kappa/d} = \zeta$. 

    Since $P$ is coprime to $s^\kappa-{s'}^{\kappa'}$, it follows that $P$ is not divisible by $s^{\kappa/d} - \zeta {s'}^{\kappa'/d}$. Hence, there exist $t, a,b,a',b'\in\Qb$ with $t\neq 0$ and $P(\zeta't^{\kappa'},t^{\kappa},a,b,a',b')\ne 0$. 
    
    For $s\in\Qb\t$ and $a,b\in\Qb$, define
    \[\gamma_{s,a,b} = \begin{pmatrix}
        0&0&-\sqrt s&0\\\sqrt s &0&a&0\\0&0&0&\sqrt{s}\\0&\sqrt s&-\sqrt s(b+1)&a
    \end{pmatrix}.\]
    Then, with respect to the symplectic matrix $J=\br{\begin{smallmatrix}
        	0 & 0 & 0 & 1\\
        	0 & 0 & 1 & 0\\
        	0 & -1 & 0 & 0\\
        	-1 & 0 & 0 & 0
        \end{smallmatrix}}$, the matrix $\gamma_{s, a, b}$ is an element of $\Gf(\Qb)$ with trace $a$, similitude $s$, and $\tr\std(\gamma_{s,a,b}) = b$.

    Set $(\gamma, \gamma') = (\gamma_{\zeta't^{\kappa'},a,b}, \gamma_{t^\kappa,a',b'})$. Then, by construction, we have
\[\simil(\gamma)^{\kappa/d} = (\zeta't^{\kappa'})^{\kappa/d}=\zeta t^{\kappa\kappa'/d}\]
    and
    \[\zeta\simil(\gamma')^{\kappa'/d} = \zeta(t^{\kappa})^{\kappa'/d}=\zeta t^{\kappa\kappa'/d},\]
    so $(\gamma,\gamma')\in\Gkk^\zeta(\Qb)$. Moreover,
    \[\phi(\gamma,\gamma') = P(\zeta't^{\kappa'},t^{\kappa},a,b,a',b')\ne 0.\]
    Thus, $\phi$ does not vanish on $\Gkk^\zeta$.
\end{proof}

\begin{proof}[Proof of \Cref{thm:b_p}]
    Let $\phi$ be the element
    \[\phi\:(\gamma,\gamma')\mapsto P(\simil(\gamma),\simil(\gamma'),\tr(\gamma), \tr\std(\gamma), \tr(\gamma'), \tr\std(\gamma'))\]
    of $\O(\Gkk)^{\Gkk}$. Then, by \Cref{lem:non-vanishing}, $\phi$ does not vanish on any connected component of $\Gkk$. The result follows from \Cref{thm:galois}.
\end{proof}

\section*{Statement on the use of generative AI}

The authors acknowledge the use of ChatGPT 5.6 Sol to proofread an earlier version of this article. This audit flagged numerous typos and several minor mathematical errors, which the authors subsequently addressed.

\section*{Acknowledgments}

The authors would like to thank Ralf Schmidt for helpful comments and corrections on an earlier version of this paper. We also thank Tobias Berger, Soumyadip Das, Manami Roy, and Zhining Wei for helpful conversations. We are grateful to the anonymous referee for their helpful comments and suggestions. AK was supported by ANRF (DST, Government of India) under the Start-Up Research Grant SRG/2023/001202. AW was supported by an AMS-Simons travel grant.

    \bibliography{bibliography}
	\bibliographystyle{alpha}
	
\end{document}